\documentclass[a4paper]{amsart}

\usepackage[latin1]{inputenc}
\usepackage{amsfonts}
\usepackage{amsmath}
\usepackage{amssymb}
\usepackage{amsthm}
\usepackage[american]{babel}
\usepackage[T1]{fontenc}
\usepackage[dvips]{graphicx}
\usepackage{xspace}
\usepackage{bbm}

\newcommand{\quotient}{\ensuremath{C^2(F,3) / \partial C^2(F,4)}\xspace}
\newcommand{\Fs}{\ensuremath{F^\times}\xspace}
\newcommand{\MZ}{\ensuremath{\mathbbm{Z}}\xspace}
\newcommand{\MQ}{\ensuremath{\mathbbm{Q}}\xspace}

\theoremstyle{plain} 
\begingroup          
\newtheorem{thm}{Theorem}[section]
\newtheorem{cor}[thm]{Corollary}
\newtheorem{conj}[thm]{Conjecture}
\newtheorem{lem}[thm]{Lemma}
\newtheorem{prop}[thm]{Proposition}
\newtheorem{ex}[thm]{Example}
\endgroup

\theoremstyle{definition}
\newtheorem{defn}[thm]{Definition}
\newtheorem*{Ack}{Acknowledgment}

\makeatletter
\newcounter{roem}
\renewcommand{\theroem}{\roman{roem}}
\newcommand{\c@org@eq}{}
\let\c@org@eq\c@equation
\newcommand{\org@theeq}{}
\let\org@theeq\theequation
\newcommand{\setroem}{
\let\c@equation\c@roem
\let\theequation\theroem}
\newcommand{\setarab}{
\let\c@equation\c@org@eq
\let\theequation\org@theeq}
\makeatother

\theoremstyle{remark}
\newtheorem{rem}[thm]{Remark}

\newcommand{\Kriz}{K\v r\'i\v z }
\newcommand{\specF}{\ensuremath{\mathsf{Spec}(F)}\xspace}

\author{Oliver Petras}

\title{Functional equations of the dilogarithm in motivic cohomology}
\date{\today}
\address{Fachbereich Physik, Mathematik und Informatik\\ Johannes Gutenberg -- Universit\"at Mainz \\ Staudinger Weg 9 \\ D -- 55099 Mainz}

\email{opetras@mathematik.uni-mainz.de}

\begin{document}

\begin{abstract}
We prove relations between fractional linear cycles in Bloch's integral cubical higher Chow complex in codimension two of number fields, which correspond to functional equations of the dilogarithm. These relations suffice, as we shall demonstrate with a few examples, to write down enough relations in Bloch's integral higher Chow group $CH^2(F,3)$ for certain number fields $F$ to detect torsion cycles. Using the regulator map to Deligne cohomology, one can check the non-triviality of the torsion cycles thus obtained.

Using this combination of methods, we obtain explicit higher Chow cycles generating the integral motivic cohomology groups of some number fields.
\end{abstract}

\maketitle

\section{Introduction}
Computing integral motivic cohomology groups $H_{\mathcal{M}}^{p,q}(X,\MZ)$ of a smooth scheme $X$ over some field $\mathbbm{k}$ explicitly is a very hard task in general. Even rationally, there are many examples of schemes with infinite dimensional motivic cohomology spaces $H_{\mathcal{M}}^{p,q}(X,\MZ)\otimes_\MZ\MQ$. But at least for rings of integers $\mathcal{O}_F$ for number fields $F$, where motivic cohomology is only a refinement of algebraic $K$--theory, one has the following theorem:
\begin{thm}[Borel \cite{Bo77}]
 $$dim_{\mathbbm{Q}}(K_n(\mathcal{O}_F)\otimes\mathbbm{Q}) = \begin{cases} 1, & n=0, \\ 0, & n\geq 2 \textrm{ even,}\\ r_1+r_2-1, & n=1,\\ r_2, & n = 2m-1, m \textrm{ even,} \\ r_1+r_2, & n = 2m-1, m \textrm{ odd,}\end{cases}$$ where $r_1$ (resp. $r_2$) denotes the number of real (resp. not conjugate complex) embeddings of $F$ into $\mathbbm{C}$.
\end{thm}
Since $K_{2n-1}(\mathcal{O}_F)\cong K_{2n-1}(F)$ for $n\geq 2$ \cite{Wei}, at least for number fields, where the algebraic $K$--groups are finitely generated abelian groups, it should be possible to find explicit generators of the integral motivic cohomology. This can be done using their connections to other, better understood, groups.

First, Voevodsky has shown \cite{Voe} that for a smooth quasi-projective variety $X$ over some field $\Bbbk$ we have $H_{\mathcal{M}}^{p,q}(X,\MZ) \cong CH^{q}(X,2q-p)$, where the latter groups denote Bloch's higher Chow groups \cite{bloch:1986}. We shall investigate $CH^2(F,3):=CH^2(\specF,3)$ for number fields $F$. By results of Suslin \cite{Sus86} it is known that $CH^2(F,3)\cong K^{ind}_3(F)$, where the right-hand side denotes the indecomposable part of Quillen's algebraic $K$--group $K_3(F)$ defined as its quotient by the image of Milnor's $K$--group $K_3^{M}(F)$. 

Secondly, Bloch and \Kriz have used the cubical version of higher Chow groups in their work on mixed Tate motives to define a natural homomorphism $$\rho_2:\MZ[\Fs]\rightarrow Z^2(F,3)$$ from the group ring of $F^\times$ to the cycle complex of admissible cycles of codimension two in the affine $3$-cube over $F$ (cf. \cite{BK:1995}). By the work of Gangl and M\"uller-Stach \cite{HGaSM99}, this map induces a well-defined homomorphism $$\overline{\rho}_2\otimes \MQ: B_2(F)\otimes\MQ\longrightarrow CH^2(F,3)\otimes\MQ,$$ which is supposed to induce a rational isomorphism between the so-called Bloch group $B_2(F)$ and the higher Chow group of $F$. We recall here that the Bloch group is roughly given by the subquotient of the free abelian group $\mathbbm{Z}[\Fs]$ modulo relations reflecting functional equations of the dilogarithm \cite{Sus86}.

Gangl's and M\"uller-Stach's method to prove the well-definedness of the map $\overline{\rho}_2\otimes \MQ$ was to prove the relations among the images of $(\overline{\rho}_2\otimes \MQ)(a), a\in \Fs$, i.e. to find appropriate cycles, in Bloch's rational higher Chow complex corresponding to the functional equations of the dilogarithm which define the Bloch group $B_2(F)$.

In this article, we shall complete the relations from \cite{HGaSM99} in Bloch's rational higher Chow complex to relations among the images of $\overline{\rho}_2(a), a\in \Fs,$ in Bloch's integral higher Chow complex. This enables us to write down enough relations, i.e. cycles, in the integral higher Chow groups in codimension two of some number fields to find torsion cycles. With the help of the generalized Abel -- Jacobi map of \cite{KLM04} we can also determine the non-triviality of these torsion cycles. Once we know the order of $CH^2(F,3)_{tors}$ for some number field $F$ (e.g. from \cite{Wei}), we can at least in principle use our methods to find generators among the images $\overline{\rho}_2(a)$ for specific $a\in\Fs$.

This is part of the author's doctoral thesis at the Johannes Gutenberg -- Universit\"at Mainz. Other parts of the thesis about similar computations in $CH^3(F,5)$ will appear elsewhere.

\begin{Ack}
 I most heartily would like to thank Stefan M\"uller-Stach for posing the original problem, his excellent supervision of my work and very great encouragement throughout the time in Mainz when progress was slow. Further, I thank Herbert Gangl very warmly for several discussions, the invitation to Durham, new ways of looking at my results and ideas to clarify the presentation.

Many thanks to an unknown referee who enormously helped to improve a preprint form of this article and also pointed out the reference to Hulsbergen's book.

Finally, financial support from the DFG project Mu 825/5-1 and the SFB/TR 45 is greatly acknowledged.
\end{Ack}

\paragraph*{\textit{Notation}}
 In the text, we let $\mathbbm{k}$ be an arbitrary field, while $F$ always denotes an infinite field, e.g. a number field. All tensor products are taken in the category of $\MZ$-modules.

\section{Bloch's higher Chow groups}
Let us begin by recalling the definition of Bloch's higher Chow groups. Since there are many good expositions in literature, we only consider the cubical version keeping in mind that Levine \cite{levine:1994} established a quasiisomorphism to the ``original'' simplicial version due to Bloch \cite{bloch:1986}.

We let $\mathbbm{k}$ be a field and $$\Box^n_\mathbbm{k}=(\mathbbm{P}^1_{\mathbbm{k}}\setminus \{1\})^n$$ with coordinates $(z_1,\ldots,z_n)$ be the algebraic standard cube with
 $2^n$ faces of codimension $1$: $$\partial\Box^n_\mathbbm{k}=\bigcup_{i=1}^n\{(z_1,\ldots,z_n)\in\Box^n_\mathbbm{k}|
z_i\in\{0,\infty\}\}$$ and faces of codimension $k$ $$\partial^k\Box^n_\mathbbm{k}=\bigcup_{i_1<\ldots <i_k}\{(z_1,\ldots,z_n)\in\Box^n_\mathbbm{k}|
z_{i_1},\ldots,z_{i_k}\in\{0,\infty\}\}.$$
In case the field $\mathbbm{k}$ is clear or irrelevant, we shall drop the subscript in the rest of the article. We now let $X$ be a smooth quasi-projective variety over $\mathbbm{k}$ and then write $Z^p(X,n)=c^p(X,n) / d^p(X,n)$ for the quotient of the free abelian group $c^p(X,n)$ generated by integral closed algebraic subvarieties of codimension $p$ in $X\times \Box^n_\mathbbm{k}$ which are admissible (i.e. meeting all faces of all codimensions in codimension $p$ again -- or not at all) modulo the subgroup $d^p(X,n)$ of degenerate cycles (i.e. pull-backs of $X\times facets,$ where a facet is a component of $\partial\Box^n$ by coordinate projections $\Box^n\rightarrow\Box^{n-1}$). These groups form a simplicial abelian group:
\renewcommand{\arraystretch}{0.3}
$$\ldots Z^p(X,3) \begin{array}{l}\rightarrow \\ \rightarrow \\ \rightarrow \\ \rightarrow \end{array} Z^p(X,2) \begin{array}{l} \rightarrow \\ \rightarrow \\ \rightarrow \end{array} Z^p(X,1)\begin{array}{l} \rightarrow \\ \rightarrow \end{array} Z^p(X,0).$$
\renewcommand{\arraystretch}{1}
\begin{defn}
Bloch's higher Chow groups $CH^p(X,n)$ are the homotopy or equivalently homology groups of the above complex with respect to Bloch's boundary map given by $$\partial_B=\sum_i (-1)^{i-1}(\partial_i^0-\partial_i^\infty),$$ where $\partial_i^0,\partial_i^\infty$ denote the restriction maps to the faces $z_i=0$ resp. $z_i=\infty$: $$CH^p(X,n):=\pi_n(Z^p(X,\bullet))=H_n(Z^p(X,\bullet),\partial_B).$$
\end{defn}

\begin{thm}[\cite{FrieVoe99}]
Assume that a field $\mathbbm{k}$ admits resolution of singularities and let $X$ be a smooth quasi-projective variety over $\mathbbm{k}$. Then Bloch's higher Chow groups are isomorphic to the motivic cohomology groups: $$CH^p(X,n)\cong H^{2p-n,p}(X,\MZ).$$
\end{thm}

The higher Chow groups satisfy several formal properties as expected of motivic cohomology: Localization, the homotopy axiom, they admit products for smooth varieties, and there is a regulator map to Deligne -- Beilinson cohomology, which we introduce in the next section. We also mention the following comparison theorem:

\begin{thm}[\cite{levine:1994},\cite{bloch-ss}]
 Let $X$ be a smooth, quasiprojective variety of dimension $d$ over a field $\mathbbm{k}$. Let further $gr_\gamma^q K_n(X)$ be the $q^{th}$ piece of the weight filtration of Quillen's $K$--theory of $X$. Then $$gr_\gamma^q K_n(X)\otimes \mathbbm{Z}\left[\frac{1}{(n+d-1)!}\right] \cong CH^q(X,n)\otimes \mathbbm{Z}\left[\frac{1}{(n+d-1)!}\right].$$ More generally, there is a spectral sequence $$CH^{-q}(X,-p-q) \Rightarrow K_{-p-q}(X)$$ for an equidimensional scheme $X$ over $\mathbbm{k}$ abutting to $K$--theory and inducing the above isomorphism after tensoring with $\mathbbm{Q}$.
\end{thm}

Thus, we can transfer the information from Bloch's higher Chow groups to information about the algebraic $K$--groups, which are in general almost impossible to describe explicitly.

\begin{conj}
 If $\mathbbm{k}$ satisfies the rank conjecture of Suslin, e.g. if $\mathbbm{k}=F$ is a number field, the cubical higher Chow groups $CH^p(F,p+q)$, $q\geq 0$, are generated by fractional linear cycles.
\end{conj}

\begin{rem}
The latter conjecture is a theorem of Gerdes for $q=0,1$ \cite{Ger}. This is the crucial fact for the work of Gangl, M\"uller-Stach, Zhao and the present author, since the Totaro cycles are specific fractional linear cycles. Thus, it makes sense to look for generators of higher Chow groups among these.
 \end{rem}

\section{The Abel -- Jacobi map}
As we have seen, the higher Chow groups are in some way a refinement
of Quillen's algebraic $K$--groups. So it is natural to ask for a
refined regulator map from Bloch's higher Chow groups to some Bloch
-- Ogus cohomology theory. Such a regulator map to Deligne --
Beilinson cohomology was proposed by Bloch in \cite{bloch}. In this
section, we introduce Deligne cohomology
$H^\bullet_{\mathcal{D}}(X,\mathbbm{Z}(*))$ for analytic varieties
$X/\mathbbm{C}$ with associated analytic spaces $X^{an}$ which
suffices for our purposes and describe the generalized Abel --
Jacobi map of \cite{KLM04}:
\begin{equation}
\label{eq:AJ1}
\Phi_{p,n}: CH^p(X,n)\rightarrow H_{\mathcal{D}}^{2p-n}(X^{an},\mathbbm{Z}(p)).
\end{equation}

\begin{rem}
 We shall work in the analytic topology and write $\Omega_X^k$ for the sheaf of holomorphic $k$-forms on $X$. In contrast, we let $\Omega_{X}^{p,q}$ be the sheaf of $\mathcal{C}^\infty$-forms of type $(p,q)$ on the associated analytic variety $X^{an}$ and $\mathcal{D}_{X}^{p,q}$ be the sheaf of distributions on  $\mathcal{C}^\infty$-forms on $X^{an}$ of type $(m-p,m-q)$. If we are only interested in the total degree, we set $\Omega_{X}^k:=\oplus_{p+q=k}\Omega_{X}^{p,q}$ as well as $\mathcal{D}^k_X:=\oplus_{p+q=k} \mathcal{D}_X^{p,q}$. Finally, cohomology groups without subscript denote Betti cohomology groups.
\end{rem}

Deligne cohomology and its properties are nicely explained in full
generality in \cite{EV88} and \cite{Ja}. Here, we will restrict
ourselves to a very elementary exposition of the Deligne cone
complex and its homology.

\begin{rem}
In general, recall that if $\mu:A^\bullet \to B^\bullet$ is a
morphism of complexes, then the cone complex is given by
\begin{equation}
\label{eq:cone} Cone(A^\bullet \stackrel{\mu}{\longrightarrow}
B^\bullet) :=A^\bullet[1]\oplus B^\bullet
\end{equation}
with differential $\delta: A^{\bullet+1}\oplus B^\bullet\to
A^{\bullet+2}\oplus B^\bullet$ given by $\delta(a,b)=
(-da,\mu(a)+db)$.
\end{rem}

\begin{defn}
Let $X$ be a smooth projective variety over $\mathbbm{C}$ of
dimension $m$ and let $\mathbbm{A}\subseteq \mathbbm{R}$ be a
subring with $\mathbbm{A}(p):=(2\pi i)^p\mathbbm{A}\subseteq
\mathbbm{C}$. Let further $C^k(X^{an};\mathbbm{Z}(p))$ denote the
singular $\mathcal{C}^\infty$-chains of real codimension $k$ on
$X^{an}$ with coefficients in $\mathbbm{Z}(p)$, which comprise the
well-known complex $(C^\bullet(X^{an};\mathbbm{Z}(p)),d)$. Then
$$C^\bullet_{\mathcal{D}}(X^{an},\mathbbm{Z}(p)):= Cone\left\{
C^\bullet(X^{an};\mathbbm{Z}(p))\oplus F^p
\mathcal{D}^\bullet_X(X^{an})\stackrel{\epsilon-
l}{\longrightarrow}\mathcal{D}^\bullet_X(X^{an})\right\}[-1],$$
where the maps $\epsilon$ and $l$ are defined in \cite{Ja}.

For us Deligne cohomology
$H^\bullet_{\mathcal{D}}(X^{an},\mathbbm{Z}(p))$ just as the
homology of the above cone complex with respect to the differential
$\delta$.
\end{defn}

Our key application of the Deligne complex is checking whether
torsion cycles in motivic cohomology are nontrivial in Deligne
cohomology. In \cite{KLM04} Kerr, Lewis and M\"uller-Stach gave an
explicit description of a well-defined map of complexes
\begin{equation}
\label{eq:reg1} \tilde{Z}(X,-\bullet)\to
C^{2p+\bullet}_{\mathcal{D}}(X^{an},\mathbbm{Z}(p))
\end{equation}
from a quasi-isomorphic subcomplex of Bloch's higher Chow complex to Deligne cohomology complex which induces a generalized Abel -- Jacobi map from Bloch's higher Chow groups to Deligne cohomology as we shall now describe:

For any cone complex of the form \eqref{eq:cone} there is a natural long exact sequence of the form
\begin{align*}
 \ldots\to H^{*-1}(A^\bullet)&\stackrel{\beta}{\to} H^{*-1}(B^\bullet)\to H^*(Cone(\{A^\bullet\to B^\bullet\}[-1]))\to \\ &\to H^*(A^\bullet)\stackrel{\alpha}{\to} H^*(B^\bullet)\to\ldots,
\end{align*}
where in our case $ker(\alpha)=Hom_{MHS}(\mathbbm{Z}(0),H^*(X,\mathbbm{Z}(p)))$, which we may assume to vanish for $* < 2p$ since we may assume $H^*(X,\mathbbm{Z})$ torsion-free. Additionally, $coker(\beta)=Ext^1_{MHS}(\mathbbm{Z}(0),H^{*-1}(X,\mathbbm{Z}(p)))$. Thus, the map \eqref{eq:reg1} induces a map
\begin{equation}
 \label{eq:reg2}
CH^p(X,n)\to H^{2p-n}_{\mathcal{D}}(X^{an},\mathbbm{Z}(p))\cong Ext^1_{MHS}(\mathbbm{Z}(0),H^{2p-n-1}(X,\mathbbm{Z}(p))).
\end{equation}

Now, we shall make this map explicit. For this we need to recall currents on smooth varieties.



\begin{defn}
 A $d$--current on the quasiprojective variety $X$ of complex dimension $m$ is a section of the sheaf $^\prime\mathcal{D}^d_X:=\mathcal{D}^{2m-d}_X$ of distributions on $\mathcal{C}^\infty$-forms on $X^{an}$.
\end{defn}


We will associate to any meromorphic function $f\in \mathbbm{C}(X)$ the oriented $(2m-1)$-chain $T_f:=\overline{f^{-1}(\mathbbm{R}^-)}$ as in \cite[5.1]{KLM04}. The orientation is chosen so that $\partial T_f = (f) = (f)_0 - (f)_\infty$. We are interested in a certain current on $\Box^n$: Let 
\begin{align*}
  T^n & := T_{z_1}\cap\ldots\cap T_{z_n}\\
\intertext{be a topological $n$--chain and write} R^n :=
R(z_1,\ldots,z_n) & := \log(z_1)d\log(z_2)\land\ldots\land
d\log(z_n) \\ & \qquad +(-1)^{n-1}\left(2\pi
i\log(z_2)d\log(z_3)\land\ldots\land
d\log(z_n)\cdot\delta_{T_{z_1}}+ \ldots \right. \\ &\qquad\left.
\ldots+(2\pi i)^{n-1}\log(z_n)\cdot\delta_{T_{z_1}\cap\ldots\cap
T_{z_{n-1}}}\right)\in\text{ } ^\prime\mathcal{D}^{n-1}_{\Box^n_F}.
\end{align*}

\begin{rem}
Via pullback, one can pretend that $T^n, R^n$ be currents on $X\times\Box^n_F$. But then one has to assure that the cycle class $\mathcal{Z}\in CH^p(X,n)$ is in real good position, i.e. intersects the faces of the real $n$-cube in an admissible way. The subgroup $\tilde{Z}^p(X,\bullet)\subset Z^p(X,\bullet)$ mentioned above is precisely the group generated by cycles in real good position.
\end{rem}

Then we can finally state the explicit form of \eqref{eq:reg2}:
\begin{prop}[\cite{KLM04}, Sect. 5.7]
The Abel -- Jacobi map $\Phi_{p,n}$ \eqref{eq:AJ1} for a cycle class
$[\mathcal{Z}]\in CH^p(\mathsf{Spec}(\mathbbm{C}),2p-1)$ is given
by: $$ \frac{1}{(-2\pi i)^{p-1}}\int\limits_{\mathcal{Z}}R^{2p-1}
\in
H^1_{\mathcal{D}}(\mathsf{Spec}(\mathbbm{C})^{an},\mathbbm{Z}(p))\cong
\mathbbm{C}/ \mathbbm{Z}(p).$$
\end{prop}
\begin{ex}
 Consider a cycle class $[\mathcal{Z}]\in CH^2(\mathsf{Spec}(\mathbbm{C}),3)$ intersecting the real $3$--cube in an admissible way. Then we have
\begin{equation}
\label{eq:regk3}
-\Phi_{2,3}[\mathcal{Z}] = \int\limits_{\mathcal{Z}\cap T_{z_1}}\hspace{-0.5 em}\log(z_2) d\log(z_3) + 2\pi i\hspace{-1 em}\sum_{p\in \mathcal{Z}\cap T_{z_1}\cap T_{z_2}}\hspace{-1 em}\log(p)
\end{equation}
as image of the cycle class. \hfill{$\blacklozenge$}
\end{ex}

\section{The setup}
The aim of this article is the determination of explicit generators for higher Chow groups of number fields $F$. This is done in the following way: We start with $X_0=\mathsf{Spec}(F)$ and prove relations corresponding to functional equations of the dilogarithm to find torsion cycles in $CH^p(X_0,2p-1)$. Then we build an analytic variety out of $X_0$ by using one of the $r=[F:\mathbbm{Q}]=r_1+2r_2$ complex embeddings $\sigma:F\hookrightarrow\mathbbm{C}$ to pull back torsion cycles on $CH^p(F,2p-1)$ along the induced map $\mathsf{Spec}(\mathbbm{C})\to \mathsf{Spec}(F)$ and apply the Abel -- Jacobi map from the previous section to check whether our cycles are indeed nontrivial: $$CH^p(F,2p-1)\stackrel{\sigma^*}{\to}CH^p(\mathbbm{C},2p-1)\stackrel{\Phi_{p,2p-1}}{\longrightarrow}H^1_{\mathcal{D}}(\mathsf{Spec}(\mathbbm{C})^{an},\mathbbm{Z}(p))\cong\mathbbm{C}/\mathbbm{Z}(p).$$

Let us introduce some notation:
\begin{defn}
In general, given a map $\phi:(\mathbbm{P}_F^1)^n\rightarrow (\mathbbm{P}_F^1)^m$ for some $n,m\in\mathbbm{N}$, let $Z_\phi$ be the cycle $\phi_*((\mathbbm{P}_F^1)^n)\cap\Box^m$ associated to $\phi$ in the sense of \cite[sect. 1.4]{Ful}. Then let $x=(x_1,\ldots,x_n)$ and $$\left[\phi_1(x),\ldots,\phi_m(x) \right]:=Z_{(\phi_1(x),\ldots,\phi_m(x))}.$$
\end{defn}

In this paper we will be concerned with codimension two Chow groups where there are certain fractional linear algebraic cycles $C_a, a\in\Fs$, available in $Z^2(F,3)$, namely the images of $a\in \Fs$ under the map $\rho_2:\MZ[\Fs]\rightarrow Z^2(F,3)$ which are likely to provide candidates (by remark 2.5) of generators of the whole Chow group of a number field: For some $a\in\Fs$ let
 $$C_a:=Z_{(1-\frac{a}{x},1-x,x)}=\left[1-\frac{a}{x},1-x,x\right]$$ be a so-called Totaro cycle resp. Totaro curve \cite{totaro:1992}. Note that $\partial C_a=(1-a,a)\neq 0\in Z^2(F,2)$. This means that except for $a=1$ there is no closed Totaro cycle, but we will see that combinations of them are closed indeed in $Z^2(F,3)$. Note also that for some integers $n_i$ and $\mathcal{Z}:= \sum_in_iC_{a_i}\in CH^2(F,3)$ one obtains by direct computation of \eqref{eq:regk3} $$-\Phi_{2,3}[\mathcal{Z}]=\sum_in_iLi_2(a_i)\in H^1_\mathcal{D}(\mathsf{Spec}(\mathbbm{C})^{an},\mathbbm{Z}(2)).$$

\begin{rem}
In the rest of the article one especially has to care about the admissibility condition mentioned in the survey on higher Chow groups. Totaro cycles are known to be admissible. Moreover, a cycle $\mathcal{Z}=[f(x),g(x),h(x)]$ is admissible if and only if the following holds: Every zero or pole of one of the rational functions which is also a zero or pole of another one, has to be contained in the preimage of $1$ of the third function.

In the rest of this paper we have either basically used analogous cycles as Gangl and M\"uller-Stach did, assumed their admissibility or checked the admissibility of all cycles occurring. So we shall not stress this point any more.
\end{rem}

To simplify our computations in the quotient $Z^2(F,3) / \partial Z^2(F,4)$, we divide out an acyclic subcomplex of $Z^2(F,\bullet)$ consisting of cycles with a constant coordinate on the left-hand side
\begin{lem}
 The following subcomplex of $Z^2(F,\bullet)$ is acyclic:
\begin{align*}
&Z'(F,\bullet):= \\
&\ldots\rightarrow Z^1(F,1)\otimes Z^1(F,3)\rightarrow Z^1(F,1)\otimes Z^1(F,2) \rightarrow Z^1(F,1)\otimes \partial Z^1(F,2) \rightarrow 0
\end{align*}
\end{lem}
\begin{proof}
First, one checks that this is a complex. This follows at once as it
is a truncated version of $Z^1(F,\bullet)$ tensored with $Z^1(F,1)$,
which does not change homology.

The acyclicity of this subcomplex is essentially the acyclicity
result of Nart \cite{nart:1995} who explicitly constructs a
contracting homotopy. His proof carries over literally to our
integral setting. Again the acyclicity is not changed by tensoring
the whole subcomplex by $Z^1(F,1)$.
\end{proof}
\begin{defn}
 We set $C^2(F,\bullet):= Z^2(F,\bullet) / Z'(F,\bullet)$.
\end{defn}

\begin{rem}
Note that the Abel -- Jacobi map is identically zero on terms
contained in $Z'(F,\bullet)$: For $\mathcal{Z}:=[c,g(t),h(t)]\in
Z'(F,3)$ with a constant $c$ such that $\sigma(c)\notin
\overline{\mathbbm{R}^-}$, we have $|\mathcal{Z}|\cap T_c=\emptyset$
which implies $\Phi_{2,3}(\mathcal{Z})=0$. By abuse of notation,
$\Phi_{2,3}$ here denotes the map of complexes inducing the Abel --
Jacobi map in the Chow groups. So $$\hspace{-0.5
em}Z^2(F,3)\stackrel{\sigma^*}{\longrightarrow}Z^2(\mathbbm{C},3)\stackrel{-\Phi_{2,3}}{\longrightarrow}
\mathbbm{C}/\mathbbm{Z}(2)$$ induces a well-defined map
$$\hspace{-0.5 em}K^{ind}_3(F)\cong CH^2(F,3)\to
\mathbbm{C}/\mathbbm{Z}(2),$$ where we used the (nonstandard)
identification
\begin{align}
\label{eq:well}
 CH^2(F,3)\cong &H_3(C^2(F,\bullet),\partial_B)\cong \notag\\
\cong &\ker\left\{\frac{Z^2(F,3)}{\partial Z^2(F,4)+Z^1(F,1)\otimes Z^1(F,2)}\to \frac{Z^2(F,2)}{Z^1(F,1)\otimes \partial Z^1(F,2)}\right\}.
\end{align}
Thus, $-\Phi_{2,3}\circ\sigma^*$ is well-defined on \eqref{eq:well}
and it suffices to construct higher Chow cycles in \eqref{eq:well}
rather than -- as usual -- in $\ker\{\frac{Z^2(F,3)}{\partial
Z^2(F,4)}\to Z^2(F,2) \}$. Also note that this reasoning is not
possible in Gangl's and M\"uller-Stach's setting since the map of
complexes inducing the Abel -- Jacobi map is not identically zero on
their (alternating) subcomplex $\subset
Z^2(F,\bullet)\otimes\mathbbm{Q}$. This in turn implies that it is
not well-defined on their quotient. Here, we quotient by a smaller
acyclic subcomplex and do not loose any information.
\end{rem}

\begin{prop}
\label{prop:rules1}
 Let $f, g, h_1, h_2$ be rational functions of one variable $x$ such that all cycles occurring are admissible. Then the following identities hold in \quotient:\setroem
\begin{small}
\begin{align}
&\left[h_1(x)h_2(x),f(x),g(x)\right] = \left[h_1(x),f(x),g(x)\right]+\left[h_2(x),f(x),g(x)\right] \\
 &-\hspace{-0.8 em}\sum_{x_0\in div(f)}\hspace{-0.5 em}\pm\left[\frac{z-h_1(x_0)h_2(x_0)}{z-h_1(x_0)},z,g(x_0)\right]+\hspace{-0.8 em}\sum_{x_0\in div(g)}\hspace{-0.5 em}\pm\left[\frac{z-h_1(x_0)h_2(x_0)}{z-h_1(x_0)},z,f(x_0)\right], \notag\\
&\left[f(x),h_1(x)h_2(x),g(x)\right] = \left[f(x),h_1(x),g(x)\right]+\left[f(x),h_2(x),g(x)\right] \\
 &+\hspace{-0.8 em}\sum_{x_0\in div(f)}\hspace{-0.5 em}\pm\left[\frac{z-h_1(x_0)h_2(x_0)}{z-h_1(x_0)},z,g(x_0)\right], \notag\\
&\left[f(x),g(x),h_1(x)h_2(x)\right] = \left[f(x),g(x),h_1(x)\right]+\left[f(x),g(x),h_2(x)\right].
\end{align}
\end{small}
\end{prop}
\setarab
\begin{proof}
 Compute boundaries of elements of $Z^2(F,4)$: $\left[\frac{z-h_1(x)h_2(x)}{z-h_1(x)},z,f(x),g(x)\right]$, $\left[f(x),\frac{z-h_1(x)h_2(x)}{z-h_1(x)},z,g(x)\right]$, resp. $\left[f(x),g(x),\frac{z-h_1(x)h_2(x)}{z-h_1(x)},z\right]$, and keep in mind that terms with a constant in the left coordinate are contained in $Z'(F,\bullet)$.
\end{proof}

Before going on to proving something with these rules, let us examine the different types of terms in our complex $C^2(F,\bullet)$. Note also that there is another kind of term with a constant coordinate possible apart from the ones in the above proposition, namely $\left[f(x),c,g(x)\right]$ for rational functions $f, g$ and a constant $c$.
\begin{prop}
An admissible term of the form $\left[f(x),c,g(x)\right]\in C^2(F,3)$ for some M\"obius transformations $f, g$ and a constant $c\in\Fs$ can be expressed a sum of terms of the form $$Z(a,c) := \left[1-\frac{1-a}{x},c,1-x\right] = \left[\frac{x-a}{x-1},c,x\right]$$ and terms with a constant in the right coordinate.
\end{prop}
\begin{proof}
Consider a generic admissible term $\mathcal{Z}:=[f(x),c,g(x)]$. Reparametrizing $x\mapsto g^{-1}(x)$ $\mathcal{Z}$ maps onto the still admissible term $[f(g^{-1}(x)),c,x]$. Then by invoking the above proposition sufficiently often, one can factor $f(g^{-1}(x))$ into terms of the form $\frac{a_i-x}{1-x}$. Note that the denominator guarantees admissibility. Again reparametrizing $x\mapsto 1-x$, we produce several terms of the form $Z(a_i,c)$ for certain constants $a,c$ and other terms with a constant on the right.
\end{proof}

\begin{rem}
The appearance of those terms in proposition \ref{prop:rules1} is the reason for the main technical problems in the sequel. These ``lower order terms'' appear as soon as one tries to prove dilogarithmic identities in the higher Chow groups and make concrete computations with the relations quite tedious. One needs some way of eliminating them to simplify computations.
\end{rem}

\begin{lem}
The following relation holds in \quotient for rational functions $f$ and $g$ provided all of the terms are admissible:
 $$\left[f(x),g(x),c\right] = -\left[f(x),c,g(x)\right]+\hspace{-0.5 em}\sum_{x_0\in div(f)}\hspace{-0.8 em}\pm Z(c,g(x_0)).$$
\end{lem}
\begin{proof}
 The relation is just the boundary of $\left[f(x),\frac{z-c}{z-1},g(x),z\right]\in C^2(F,4)$.
\end{proof}

\begin{lem}
Let $a,b,c \in\Fs$. Then there is a cycle $\mathcal{W}\in C^2(F,4)$ whose boundary gives rise to the following relation in \quotient provided all the terms are admissible:
$$\left[f(x),c,g(x)\right]=-\left[\frac{1}{f(x)},c,g(x)\right] +\hspace{-0.5 em}\sum_{x_0\in div(g)}\hspace{-0.8 em} \pm\left(\left[\frac{z-1}{z-f(x_0)},c,z\right]
 + Z(c,f(x_0))\right).$$
\end{lem}
\begin{proof}
 Set $\mathcal{W}:=-\left[\frac{z-1}{z-f(x)},z,c,g(x)\right]$, compute its boundary to be $$\left[f(x),c,g(x)\right] +\left[\frac{1}{f(x)},c,g(x)\right]+\hspace{-0.5 em}\sum_{x_0\in div(g)}\hspace{-0.8 em}\pm\left[\frac{z-1}{z-f(x_0)},z,c\right]$$ and then use the lemma above.
\end{proof}
\begin{cor}
\label{cor:z_easy}
With the assumptions of the lemma we obtain:
$$
\left[\frac{x-a}{x-b},c,x\right] =Z(c,\frac{a}{b}), \qquad Z(a,c) = Z(c,a) \in\quotient.
$$
\end{cor}
\begin{proof}
One applies the lemma to calculate
\begin{align*}
\left[\frac{x-a}{x-b},c,x\right] &= -\left[\frac{x-b}{x-a},c,x\right]+\left[\frac{x-1}{x-\frac{a}{b}},c,x\right]+Z(c,\frac{a}{b}),\\
\intertext{multiplying the numerator and denominator of first coordinate of the second term by $b$ followed by substituting $x\mapsto xb^{-1}$, we have}
 &=-\left[\frac{x-b}{x-a},c,x\right]+\left[\frac{x-b}{x-a},c,xb^{-1}\right]+Z(c,\frac{a}{b}). \\
\intertext{Splitting the second term in the last coordinate gives the first term and one term with two constant coordinates, i.e. an negligible term contained in $d^2(F,3)$ in the sense of section 2. So}
\left[\frac{x-a}{x-b},c,x\right] &=Z(c,\frac{a}{b})
\end{align*}
For the second assertion use the lemma to obtain
\begin{equation*}
Z(a,c) = \left[\frac{x-a}{x-1},c,x\right] = - \left[\frac{x-1}{x-a},c,x\right]+\left[\frac{x-1}{x-a},c,x\right]+Z(c,a)= Z(c,a).\qedhere
\end{equation*}
\end{proof}

In summary
\begin{equation}
\label{eq:zs}
 \left[\frac{x-a}{x-b},x,c\right] = Z(c,a)-Z(c,b)-Z\left(c,\frac{a}{b}\right).
\end{equation}

Trivially, one also derives $\left[\frac{x-a}{x-b},x,c\right] = \left[\frac{x-a}{x-\frac{a}{b}},x,c\right]$ by comparing the corresponding right-hand sides. Note finally that $$\left[\frac{x-a}{x-b},x,c\right] + \left[\frac{x-b}{x-a},x,c\right] = \left[\frac{x-1}{x-\frac{a}{b}},x,c\right].$$
\begin{cor}
Let $a,b,c\in \Fs$ and assume that $F$ contains an $n^{th}$ primitive root of unity $\zeta_n$. Then the following relations hold in \quotient:
$$0=n\left[\frac{x-a}{x-b},x,\zeta_n\right] = n\left( Z(a,\zeta_n) - Z\left(\frac{a}{b},\zeta_n\right) - Z(b,\zeta_n)\right).$$
\end{cor}
\begin{proof}
The relation follows trivially from equation \eqref{eq:zs} and proposition \ref{prop:rules1}.
\end{proof}

In the sequel, we will make use of one more trick: Permuting coordinate entries:
\begin{lem}
\label{lem:permute}
Let $f, g, h$ be rational functions in one variable, and let all of the cycles be admissible. Then the following identities hold in \quotient:
 \begin{align*}
\left[f(x),g(x),h(x)\right] &= -\left[f(x),h(x),g(x)\right]+\hspace{-0.5 em}\sum_{x_0\in div(f)}\hspace{-0.8 em}\pm Z(h(x_0),g(x_0)), \\
\left[f(x),g(x),h(x)\right] &= -\left[g(x),f(x),h(x)\right] +\hspace{-0.5 em}\sum_{x_0\in div(g)}\hspace{-0.8 em}\pm Z(f(x_0),h(x_0)),\\
\left[f(x),g(x),h(x)\right] &= -\left[h(x),g(x),f(x)\right] +\hspace{-0.5 em} \sum_{x_0\in div(f)}\hspace{-0.8 em}\pm Z(h(x_0),g(x_0))\\
 &\qquad -\hspace{-0.5 em}\sum_{x_0\in div(g)}\hspace{-0.8 em}\pm Z(f(x_0),h(x_0))+\hspace{-0.5 em}\sum_{x_0\in div(h)}\hspace{-0.8 em}\pm Z(f(x_0),g(x_0)).
\end{align*}
\end{lem}
\begin{proof}
 Compute the boundary of $\left[f(x),\frac{z-g(x)}{z-1},h(x),z\right]$, respectively the one of  $\left[\frac{z-f(x)}{z-1},g(x),z,h(x)\right]$, and lastly add the boundaries of $\left[\frac{z-f(x)}{z-1},g(x),h(x),z\right]$ and $\left[\frac{z-g(x)}{z-1},h(x),z,f(x)\right]$.
\end{proof}
\begin{ex}
\label{ex:permute}
The following relation holds in \quotient:
 $$\left[1-\frac{a}{x},x,1-x\right] = -C_a + Z(a,1-a).$$ So we can define
$$\tilde{C}_a := \left[1-\frac{a}{x},1-x,x\right]-\left[1-\frac{a}{x},x,1-x\right],$$ and see that $\tilde{C}_a = 2C_a-Z(a,1-a)$, in particular $\tilde{C}_1 = 2C_1$. We shall use this variant of terms in \quotient later on. They play the role of the Rogers dilogarithm: Expressing relations for the dilogarithm in terms of these elements eliminates terms with a constant in the middle. \hfill{$\blacklozenge$}
\end{ex}

\section{Relations}
With the aid of proposition \ref{prop:rules1} we are now ready to mimic the proofs of several relations as in \cite{HGaSM99}. The ideas will be more or less the same: Starting with a suitable reparametrization of a Totaro cycle, we break this term up into pieces which can be identified or at least merged into other Totaro cycles giving a relation between certain Totaro cycles and some lower order terms.
\begin{prop}
\label{prop:rel1}
For $a\in \Fs-\{1\}$ the following identity holds in the quotient \quotient:
\begin{equation}
\label{eq:rel1}
C_a+C_{1-a}-C_1 = Z(a,1-a).
\end{equation}
\end{prop}

\begin{proof}
One easily calculates
\begin{align*}
C_a &= \left[\frac{x-a}{x},1-x,x\right] = \left[\frac{x-a}{x-1},1-x,x\right] + \left[\frac{x-1}{x},1-x,x\right] \\
 &= \left[1-\frac{1-a}{x},x,1-x\right]+ C_1.
\intertext{Next one makes use of the first relation of lemma \ref{lem:permute} to see that}
 C_a &= -C_{1-a}+Z(1-a,a) + C_1.
\end{align*}
The claim follows from corollary \ref{cor:z_easy}.
\end{proof}
\begin{rem}
 Copying the proof for $\tilde{C}_a$ from example \ref{ex:permute} instead of $C_a$, we obtain a relation without the lower order term: $\tilde{C}_a+\tilde{C}_{1-a}-\tilde{C}_1=0\in\quotient$.
\end{rem}

The following distribution relations are more interesting. Note that by \cite[sect. 1.4]{Ful} we have $nC_{a^n} = \left[1-\left(\frac{a}{z}\right)^n,1-z^n,z^n\right]$:
\begin{prop}
\label{prop:verteil}
Let $a\in \Fs$ and assume $F$ contains a primitive $n^{th}$ root of unity $\zeta_n$. Then the following relation holds in \quotient:
\begin{equation}
\label{eq:rel2}
nC_{a^n} = n^2\sum_{j=1}^nC_{\zeta_n^j a} + 2\sum_{i=2}^n\left[\frac{z-\prod_{j=1}^i(1-\zeta_n^j a)}{z-(1-\zeta_n^ia)},z,a^n\right].
\end{equation}
\end{prop}

\begin{proof}
We prove the formula for $n=2$, the general case follows by repetition of similar arguments, mainly proposition \ref{prop:rules1}:
\begin{equation}
\begin{split}
\label{eq:mist}
2C_{a^2} &= \left[1-\left(\frac{a}{z}\right)^2,1-z^2,z^2\right] \overset{(i)}{=} \left[1-\frac{a}{z},1-z^2,z^2\right]+\left[1+\frac{a}{z},1-z^2,z^2\right] \\
 &\overset{(ii)}{=}2\left[1-\frac{a}{z},1-z,z^2\right]+2\left[1+\frac{a}{z},1-z,z^2\right] \\
 &\qquad +\left[\frac{z-(1-a^2)}{z-(1-a)},z,a^2\right] + \left[\frac{z-(1-a^2)}{z-(1+a)},z,a^2\right] \\
 &\overset{(iii)}{=} 4C_a + 4C_{-a} + \left[\frac{z-(1-a^2)}{z-(1-a)},z,a^2\right] + \left[\frac{z-(1-a^2)}{z-(1+a)},z,a^2\right],
\end{split}
\end{equation}
and the assertion in case $n=2$ follows from the relation $[\frac{z-a}{z-b},z,c]=[\frac{z-a}{z-\frac{a}{b}},z,c]$ in $\quotient$ for $a,b,c\in \Fs$, so $2C_{a^2}=4C_a+4C_{-a}+2\left[\frac{z-(1-a^2)}{z-(1-a)},z,a^2\right].$
\end{proof}
\begin{rem}
 Again, we can produce an analogous relation in terms of the $\tilde{C}_a$: with still no $Z$-terms, but with more terms with a constant in the right coordinate.
\end{rem}

Now let us turn to the five-term relation. We shall prove it with several ways of simplifying the extra terms. The first one is a refined version of the one Gangl and M\"uller-Stach proved. After that, one can use the relations from the previous subsection and some symmetry considerations to simplify it modulo $2$-torsion.
\begin{prop}
 Let $a,b\in\Fs - \{1\}$ such that $a\neq b, 1-b$. Then the following relation holds in \quotient:
\begin{equation}
\begin{split}
 \label{eq:five}
V_{a,b} := &C_{\frac{a(1-b)}{b(1-a)}} - C_{\frac{1-b}{1-a}} + C_{1-b} - C_{\frac{a}{b}} + C_a - Z\left(\frac{1}{b},\frac{1}{1-a}\right) -\left[\frac{z-1}{z-b},z,1-b\right]\\
 &\quad +\left[\frac{z-\frac{b-a}{b(1-a)}}{z-\frac{b-a}{1-a}},z,\frac{1-b}{1-a}\right]
 + \left[\frac{z-1}{z-(1-a)},z,a\right]
 +\left[\frac{z-\frac{b-a}{b(1-a)}}{z-\frac{b-a}{b}},z,\frac{a}{b}\right] = 0.
\end{split}
\end{equation}
\end{prop}

\begin{proof}
 We mimic the proof of \cite{HGaSM99} making use of the basic proposition \ref{prop:rules1} We shall use the tags numbering of the formulas in  proposition \ref{prop:rules1}. Let us start with a reparametrization of $C_{\frac{a(1-b)}{b(1-a)}}$.
\begin{align*}
 C_{\frac{a(1-b)}{b(1-a)}} &= \left[\frac{b-t}{b(1-t)},\frac{t-a}{t(1-a)},\frac{a(1-t)}{t(1-a)}\right] \\ &
 \overset{(iii)}{=} \left[\frac{b-t}{b(1-t)},\frac{t-a}{t(1-a)},\frac{1-t}{1-a}\right]+\left[\frac{b-t}{b(1-t)},\frac{t-a}{t(1-a)},\frac{a}{t}\right] \\
 &\overset{(i),(ii)}{=} \left[\frac{b-t}{1-t},\frac{t-a}{t(1-a)},\frac{1-t}{1-a}\right] + \left[\frac{z-1}{z-\frac{1}{b}},z,\frac{1}{1-a}\right] +\left[\frac{b-t}{b(1-t)},\frac{t-a}{t},\frac{a}{t}\right]\\
 &\qquad +\left[\frac{b-t}{b(1-t)},\frac{1}{1-a},\frac{a}{t}\right]+\left[\frac{z-\frac{b-a}{b(1-a)}}{z-\frac{b-a}{b}},z,\frac{a}{b}\right]
  - \left[\frac{z-1}{z-(1-a)},z,a\right] \\
 &\overset{(i),(ii)}{=} \left[\frac{b-t}{1-t},\frac{t-a}{1-a},\frac{1-t}{1-a}\right] +\left[\frac{b-t}{1-t},\frac{1}{t},\frac{1-t}{1-a}\right]+\left[\frac{b-t}{b},\frac{t-a}{t},\frac{a}{t}\right]\\
&\qquad +\left[\frac{1}{1-t},\frac{t-a}{t},\frac{a}{t}\right] +\left[\frac{b-t}{b(1-t)},\frac{1}{1-a},\frac{a}{t}\right] +\left[\frac{z-\frac{b-a}{b(1-a)}}{z-\frac{b-a}{1-a}},z,\frac{1-b}{1-a}\right] \\
&\qquad + \left[\frac{z-1}{z-\frac{1}{b}},z,\frac{1}{1-a}\right] +\left[\frac{z-\frac{b-a}{b(1-a)}}{z-\frac{b-a}{b}},z,\frac{a}{b}\right]
-\left[\frac{z-1}{z-(1-a)},z,a\right].
\intertext{We obtain after some inversions as in Gangl's and M\"uller-Stach's proof:}
&= C_{\frac{1-b}{1-a}} - C_{1-b} + \left[\frac{b-t}{1-t},\frac{1}{t},\frac{1}{1-a}\right] + C_{\frac{a}{b}} - C_a \\
 &\qquad +\left[ \frac{b-t}{b(1-t)},\frac{1}{1-a},\frac{a}{t}\right] +\left[\frac{z-\frac{b-a}{b(1-a)}}{z-\frac{b-a}{1-a}},z,\frac{1-b}{1-a}\right] +\left[\frac{z-1}{z-b},z,1-b \right]\\
&\qquad + \left[\frac{z-1}{z-\frac{1}{b}},z,\frac{1}{1-a}\right]
+\left[\frac{z-\frac{b-a}{b(1-a)}}{z-\frac{b-a}{b}},z,\frac{a}{b}\right]-\left[\frac{z-1}{z-(1-a)},z,a\right].
\intertext{Below, we also show that
$$\left[\frac{b-t}{1-t},\frac{1}{t},\frac{1}{1-a}\right] +
\left[\frac{b-t}{b(1-t)},\frac{1}{1-a},\frac{a}{t}\right] =
Z\left(\frac{1}{b},\frac{1}{1-a}\right) -
\left[\frac{z-1}{z-\frac{1}{b}},z,\frac{1}{1-a}\right]$$ to conclude}
C_{\frac{a(1-b)}{b(1-a)}}&= C_{\frac{1-b}{1-a}} - C_{1-b} + C_{\frac{a}{b}} - C_a + Z\left(\frac{1}{b},\frac{1}{1-a}\right)+\left[\frac{z-1}{z-b},z,1-b\right] \\
 &\qquad -\left[\frac{z-\frac{b-a}{b(1-a)}}{z-\frac{b-a}{1-a}},z,\frac{1-b}{1-a}\right]
 -\left[\frac{z-\frac{b-a}{b(1-a)}}{z-\frac{b-a}{b}},z,\frac{a}{b}\right]-\left[\frac{z-1}{z-(1-a)},z,a\right].
\end{align*}
So it remains to use prop. \ref{prop:rules1} and lemma \ref{lem:permute} to show the following:
\begin{align*}
\left[\frac{b-t}{b(1-t)},\frac{1}{1-a},\frac{a}{t}\right] &= \left[\frac{b-t}{1-t},\frac{1}{1-a},\frac{1}{t}\right] -\left[\frac{z-1}{z-\frac{1}{b}},z,\frac{1}{1-a}\right]\\
 &= -\left[ \frac{b-t}{1-t},\frac{1}{t},\frac{1}{1-a}\right] +
 Z(\frac{1}{b},\frac{1}{1-a})-\left[\frac{z-1}{z-\frac{1}{b}},z,\frac{1}{1-a}\right].
\end{align*}
Adding the other term $\left[
\frac{b-t}{1-t},\frac{1}{t},\frac{1}{1-a}\right]$ gives the desired
result.
\end{proof}
\begin{cor}
 For $a\in \Fs - \{1\}$ terms of the form $\left[\frac{z-1}{z-a},z,1-a \right]$ are $2$-torsion in \quotient.
\end{cor}
\begin{proof}
 Compute $0= V_{a,b} - V_{1-b,1-a}$ for $a,b\in \Fs - \{1\}, a\neq b,1-b$ to obtain $$2\left[\frac{z-1}{z-a},z,1-a \right]=  2\left[\frac{z-1}{z-(1-b)},z,b \right].$$ Specializing $b=-1$, we obtain the result for $a\neq -1,2$. But for $a=2$ the result holds trivially, and for $a=-1$, one has by proposition \ref{prop:rules1}
\begin{align*}
2\left[\frac{z-1}{z+1},z,2 \right]&=\left[\frac{z-1}{z+1},z,2 \right]-\left[\frac{z-1}{z+1},z,\frac{1}{2} \right]\\
&\stackrel{\eqref{eq:zs}}{=}2Z(-1,2)-2Z\left(-1,\frac{1}{2}\right)=2\left[\frac{z-1}{z-2},z,-1 \right]
\end{align*}
by the symmetry of $Z(a,b)$ and again \eqref{eq:zs}. But this expression vanishes as shown above.
\end{proof}
\begin{cor}
 Let $a,b\in\Fs - \{1\}$ such that $a\neq b, 1-b$. Then the following relation holds in \quotient:
\begin{align}
\label{eq:five2}
C_{\frac{a(1-b)}{b(1-a)}} - &C_{\frac{1-b}{1-a}} + C_{1-b} - C_{\frac{a}{b}} + C_a - Z(b,1-a) +\left[\frac{z-\frac{b-a}{b(1-a)}}{z-\frac{b-a}{1-a}},z,\frac{1-b}{1-a}\right]\notag \\
 &+ \left[\frac{z-1}{z-a},z,1-b\right]+\left[\frac{z-\frac{b-a}{b(1-a)}}{z-\frac{b-a}{b}},z,\frac{a}{b}\right]\\ &+\left[\frac{z-1}{z-(1-a)},z,a \right]-\left[\frac{z-1}{z-b},z,1-b\right]  = 0.\notag
\end{align}
\end{cor}
\begin{proof}
 Use the proposition and $Z\left(\frac{1}{b},\frac{1}{1-a}\right)=Z(b,1-a)+\left[\frac{z-1}{z-a},z,1-b \right].$
\end{proof}
\begin{cor}
 Let $a,b\in\Fs - \{1\}$ such that $a\neq b, 1-b$. Then the following relation holds in \quotient:
\begin{equation}
\begin{split}
\label{eq:five_new}
&2V'(a,b):= 2C_{\frac{a(1-b)}{b(1-a)}} - 2C_{\frac{1-b}{1-a}} + 2C_{1-b} - 2C_{\frac{a}{b}} + 2C_a - 2Z(b,1-a) \\
 &+2\left[\frac{z-\frac{b-a}{b(1-a)}}{z-\frac{b-a}{1-a}},z,\frac{1-b}{1-a}\right]
 +2\left[\frac{z-1}{z-a},z,1-b \right]+2\left[\frac{z-\frac{b-a}{b(1-a)}}{z-\frac{b-a}{b}},z,\frac{a}{b}\right] = 0.
\end{split}
\end{equation}
\end{cor}
\begin{proof}
 Combine corollaries 5.6 and 5.7.
\end{proof}

\begin{cor}
 Let $a,b\in \Fs - \{1\}$, $a\neq b,1-b$. Then $$2\left[\frac{z-1}{z-a},z,1-b \right]=2\left[\frac{z-1}{z-(1-b)},z,a \right] \in \quotient.$$
\end{cor}
\begin{proof}
Compute $2V'(a,b)-2V'(1-b,1-a)$.
\end{proof}
\begin{rem}
 Using the $\tilde{C}_a$-terms from example \ref{ex:permute} instead of $C_a$, we can again get rid of the $Z$-term, but for the price of several more terms with a constant in the right coordinate. We shall just use the relation just stated.
\end{rem}
Now we come to an inversion formula, which will be valuable afterwards.
\begin{prop}
\label{prop:inv}
For $c=a/b\in\Fs$ such that $a,b\neq 1, a\notin\left\{ b,1-b,\frac{b}{b-1}\right\}$, the following inversion relation holds in \quotient:
\begin{align*}
2&\left(  C_{c} + C_{\frac{1}{c}} - 2C_1 \right) = Z(a,1-a)+Z(b,1-b)+Z\left(\frac{1}{a},1-\frac{1}{a}\right)+Z\left(\frac{1}{b},1-\frac{1}{b}\right) \notag \\
 & - Z\left(\frac{1}{b},\frac{1}{1-a}\right) - Z\left(b,\frac{a}{a-1}\right) - Z\left(\frac{1}{a},\frac{1}{1-b}\right) - Z\left(a,\frac{b}{b-1}\right) \\
 \intertext{plus some additional monodromy terms}
 & -\left[\frac{z-\frac{a-b}{1-b}}{z-\frac{a-b}{a(1-b)}},z,\frac{b(1-a)}{a(1-b)}\right] - \left[\frac{z-1}{z-(1-a)},z,a\right] -\left[\frac{z-\frac{a-b}{1-b}}{z-\frac{b-a}{b}},z,\frac{a}{b}\right]+\left[\frac{z-1}{z-b},z,1-b \right]\notag \\
 & -\left[\frac{z-\frac{a-b}{a(1-b)}}{z-\frac{a-b}{1-b}},z,\frac{1-a}{1-b}\right] - \left[\frac{z-1}{z-(1-\frac{1}{a})},z,\frac{1}{a} \right] - \left[\frac{z-\frac{a-b}{a(1-b)}}{z-\frac{a-b}{a}},z,\frac{b}{a}\right]+\left[\frac{z-1}{z-\frac{1}{b}},z,1-\frac{1}{b} \right]\notag \\
 & -\left[\frac{z-\frac{b-a}{1-a}}{z-\frac{b-a}{b(1-a)}},z,\frac{a(b-1)}{b(a-1)}\right] -\left[\frac{z-1}{z-(1-b)},z,b\right]-\left[\frac{z-\frac{a-b}{1-a}}{z-\frac{a-b}{a}},z,\frac{b}{a}\right]+\left[\frac{z-1}{z-a},z,1-a \right]\notag \\
 & -\left[\frac{z-\frac{b-a}{b(1-a)}}{z-\frac{b-a}{1-a}},z,\frac{1-b}{1-a}\right] - \left[\frac{z-1}{z-(1-\frac{1}{b})},z,\frac{1}{b}\right] -\left[\frac{z-\frac{b-a}{b(1-a)}}{z-\frac{b-a}{b}},z,\frac{a}{b}\right]+\left[\frac{z-1}{z-\frac{1}{a}},z,1-\frac{1}{a} \right].\notag
\end{align*}
\end{prop}
\begin{proof}
Just as in the proof of \cite[Thm. 2.4]{HGaSM99}. One only has to collect all extra terms.
\end{proof}
If we use our results modulo $2$-torsion, then we can improve the result:
\begin{prop}
 For $c=a/b\in\Fs$ such that $a,b\neq 1, a\notin\left\{ b,1-b,\frac{b}{b-1}\right\}$, the following inversion relation holds in the quotient \quotient:
\begin{equation}
\begin{split}
\label{eq:inv2}
&0=4\left(C_{c}+ C_{\frac{1}{c}}-2C_1\right) - 2Z(b,1-a) - 2Z(a,1-b)\\
&- 2Z\left(\frac{1}{b},1-\frac{1}{a}\right) - 2Z\left(\frac{1}{a},1-\frac{1}{b}\right)+2Z(b,1-b)+2Z\left(\frac{1}{b},1-\frac{1}{b}\right)\\
 &+2\left[\frac{z-\frac{b-a}{b(1-a)}}{z-\frac{b-a}{1-a}},z,\frac{1-b}{1-a}\right]
 +2\left[\frac{z-1}{z-a},z,1-b \right]+2\left[\frac{z-\frac{b-a}{b(1-a)}}{z-\frac{b-a}{b}},z,\frac{a}{b}\right]\\
 &+2\left[\frac{z-\frac{a-b}{a(1-b)}}{z-\frac{a-b}{1-b}},z,\frac{1-a}{1-b}\right]
 +2\left[\frac{z-1}{z-b},z,1-a \right]+2\left[\frac{z-\frac{a-b}{a(1-b)}}{z-\frac{a-b}{a}},z,\frac{b}{a}\right]\\
 &+2\left[\frac{z-\frac{a-b}{a-1}}{z-\frac{a-b}{b(a-1)}},z,\frac{a(1-b)}{b(1-a)}\right]
 +2\left[\frac{z-1}{z-\frac{1}{a}},z,1-\frac{1}{b} \right]+2\left[\frac{z-\frac{a-b}{a-1}}{z-\frac{a-b}{a}},z,\frac{b}{a}\right]\\
&+2\left[\frac{z-\frac{b-a}{b-1}}{z-\frac{b-a}{a(b-1)}},z,\frac{b(1-a)}{a(1-b)}\right]
 +2\left[\frac{z-1}{z-\frac{1}{b}},z,1-\frac{1}{a} \right]+2\left[\frac{z-\frac{b-a}{b-1}}{z-\frac{b-a}{b}},z,\frac{a}{b}\right].
\end{split}
\end{equation}
\end{prop}
\begin{proof}
 Compute $2V'(a,b)+2V'(\frac{1}{a},\frac{1}{b})-(2C_b+2C_{1-b}-2C_1-Z(b,1-b))-(2C_{\frac{1}{b}}+2C_{1-\frac{1}{b}}-2C_1-2Z(\frac{1}{b},1-\frac{1}{b})$ and then add the corresponding expression with $a$ and $b$ interchanged. Then use \eqref{eq:zs} when needed.
\end{proof}
\begin{rem}
 It is also possible to use the multiplicativity of the $Z$-terms \eqref{eq:zs} in order to have one $Z$-term, $Z(ab,ab)$, only at the end resembling the functional equation of the dilogarithm, but this again produces several terms with a constant in the right coordinate. Therefore, we shall not make this explicit.
\end{rem}

\section{Application to number fields}
In this section, we use the relations from the preceding section in combination with the Abel -- Jacobi map introduced earlier to find explicit generators for the Chow groups of some number fields. The strategy is the following:

We shall use the relations from the last section to find (cyclotomic) torsion cycles in the Chow group and to put an upper bound on their order. Then we use the well-definedness of the Abel -- Jacobi map (i.e. the fact that $\Phi_{2,3}(\mathcal{Z})\neq 0 \Rightarrow \mathcal{Z}\neq 0$ for some $[\mathcal{Z}]\in CH^2(F,3)$) to put a lower bound on the order of a torsion cycle. In case both bound coincide, we have determined the order of the cycle in question.

We heavily make use of the following result from \cite{Wei}, where we also take the exact orders of the torsion parts of the algebraic $K$--groups in our examples from:
\begin{thm}[Thm. 0.1]
Let $F$ be a number field with $r_1$ real and $r_2$ conjugate pairs of complex places, and let $\mathcal{O}_S$ be a ring of $S$-integers. Then $K_3(\mathcal{O}_S)\cong K_3(F)$. Further
$$
K_3(F)\cong \begin{cases} \MZ^{r_2}\oplus \MZ/ w_2(F)\MZ, & F \text{ is totally imaginary,} \\
 \MZ^{r_2}\oplus\MZ/ 2w_2(F)\MZ\oplus (\MZ/2\MZ)^{r_1-1}, & F \text{ has a real embedding,} \end{cases}$$
\end{thm}
where the integer $w_2$ is defined as follows: Let $\overline{F}$ be a separable closure of $F$ and $\mathcal{G}=Gal(\overline{F}/F)$ be the absolute Galois group. The abelian group $\mu$ of all roots of unity in $\overline{F}$ is known to be a $\mathcal{G}$-module. We write $\mu(2)$ for the abelian group $\mu$ made into a $\mathcal{G}$-module by letting $g\in\mathcal{G}$ act as $\zeta\mapsto g^2(\zeta)$. In case $F$ is a global or local field, it is proved in \cite[2.3.1]{Wei} that the group $\mu(2)^{\mathcal{G}}$ is a finite group with its order denoted by $w_2(F)$.

As remarked in \cite[p.7]{Wei}, the integer $w_2(F)$ is always divisible by $24$ and can be computed rather easily. In loc. cit. it is remarked that Bass and Tate have shown that $$K_3^M(F)\cong(\MZ/2\MZ)^{r_2}$$ and by Results of Merkurjev and Suslin, one knows that this group injects into the Quillen $K$--group, so that one can in principle compute the indecomposable $K_3$ of a given number field abstractly. As an application of our relations from the last section, we determine explicit generators of the $K$--groups, i.e. higher Chow groups of some number fields.

\begin{rem}
As the referee pointed out, one has to note that the Abel -- Jacobi map as we introduced it is insufficient to capture linear dependence of non-torsion cycles. For this one needs to take $X:=\mathsf{Spec}(F)\otimes_{\mathbbm{Q}}\mathbbm{C}=\coprod_{\sigma:F\hookrightarrow\mathbbm{C}}(\mathsf{Spec}(\mathbbm{C})),$ which is simply given by $r=[F:\mathbbm{Q}]$ points and pull back cycles $\mathcal{Z}$ under each embedding (for the complex ones maybe just one of two conjugate \cite{Hul}) to get a cycle in $CH^3(X,3)=\oplus_\sigma CH^2(\mathbbm{C},3)$ and then apply $-\Phi_{2,3}$ to each of the (pulled back) cycles to get a vector in $(\mathbbm{C}/\mathbbm{Z}(2))^{\oplus r}$. It may at least be conjectured that the map $CH^2(F,3)\to (\mathbbm{C}/\mathbbm{Z}(2))^{\oplus r}$ is injective.
\end{rem}

\subsection{$CH^2(\MQ,3)$:} By the work of Lee and Szczarba (cf. \cite{LS:76}), we know that $CH^2(\mathbbm{Q},3) \cong K_3^{ind}(\mathbbm{Q})\cong (\mathbbm{Z} / 48\mathbbm{Z}) / (\mathbbm{Z}/2\mathbbm{Z})\cong
\mathbbm{Z}/ 24\mathbbm{Z}$.

From the distribution relation \eqref{eq:rel2} for $n=2$ we know $2C_1= -4C_{-1}\in CH^2(\MQ,3)$. We now use a specialization of the inversion relation to obtain a second relation between
$C_1$ and $C_{-1}$: We set $b=-a$. Noting
that then for each five-term relation the last two extra terms are $2$-torsion, we use twice the inversion relation from proposition \ref{prop:inv} and forget
these terms at once:
\begin{equation}
\label{eq:inv3}
\begin{split}
8&C_{-1} -8C_1= \\
&= 2\left(Z(-a,\frac{a}{a-1}) +Z\left(-\frac{1}{a},\frac{1}{1-a}\right) +Z\left(\frac{1}{a},\frac{1}{1+a}\right) + Z\left(a,\frac{a}{1+a}\right)\right. \\
 &\quad +Z\left(-\frac{1}{a},1+\frac{1}{a}\right)+Z\left(\frac{1}{a},1-\frac{1}{a}\right)+Z\left(-a,1+a\right)+Z\left(a,1-a\right)\\
 &\quad -\left[\frac{z-\frac{2a}{1+a}}{z-\frac{2}{1+a}},z,\frac{1-a}{1+a}\right]
 -\left[\frac{z-\frac{2}{1+a}}{z-\frac{2a}{1+a}},z,\frac{1-a}{1+a}\right] \\
 &\quad\left. -\left[\frac{z-\frac{2a}{a-1}}{z-\frac{2}{1-a}},z,\frac{1+a}{1-a}\right] -\left[\frac{z-\frac{2}{1-a}}{z-\frac{2a}{a-1}},z,\frac{1+a}{1-a}\right]\right).
\end{split}
\end{equation}
Since the admissible cycle $\left[\frac{z-1}{z-\frac{x-a}{x-b}},z,x,c \right]\in C^2(F,4)$ bounds to $$\left[\frac{x-a}{x-b},x,c \right]+\left[\frac{x-b}{x-a},x,c \right] - \left[\frac{x-1}{x-\frac{a}{b}},x,c \right]=0\in \quotient,$$ we can further simplify:
\begin{equation}
\label{eq:stuck}
\begin{split}
8&C_{-1} - 8C_1 = \\
\quad &=2\left(Z\left(-a,\frac{a}{a-1}\right) +Z\left(-\frac{1}{a},\frac{1}{1-a}\right) +Z\left(\frac{1}{a},\frac{1}{1+a}\right) + Z\left(a,\frac{a}{1+a}\right)\right. \\
 &\qquad +Z\left(-\frac{1}{a},1+\frac{1}{a}\right)+Z\left(\frac{1}{a},1-\frac{1}{a}\right)+Z\left(-a,1+a\right)+Z\left(a,1-a\right)\\
 &\qquad\left. +\left[\frac{z-1}{z-a},z,1+a\right] -\left[\frac{z-1}{z+a},z,1-a\right]\right).
\end{split}
\end{equation}

Unfortunately we do not have any information about the orders of the terms on the right-hand side in general. But we have a freedom of choice for the parameter $a$. If we chose the constants on the right-hand side of the extra terms resp. one of the constants in the $Z$--terms to be a root of unity, we could control the order of the extra terms. In order to be able to choose suitable constants, the following theorem due to Levine helps:

\begin{prop}[\cite{Lev89}, cor. 4.6]
\label{prop:levine}
Let $E$ be an arbitrary field and $F$ an extension of $E$. Then the map $K_3^{ind}(E) \rightarrow K_3^{ind}(F)$ induced by the inclusion $E \hookrightarrow F$ is injective.
\end{prop}

Since we know that $K_3^{ind}(E) \cong CH^2(E,3)$, this is the result can be applied in the following way: Returning to equation \eqref{eq:stuck}, we specialize further by setting $a=i:=\sqrt{-1}$ to bound the order of $C_1\in CH^2(\mathbbm{Q}(i),3)$. This bound cannot be higher in $CH^2(\mathbbm{Q},3)$ because of Levine's result quoted above on the injectivity of the map between the corresponding Chow groups.
\begin{align*}
8C_{-1} -8C_1 &= 2Z\left(i,\frac{1+i}{2}\right)+2Z\left(-i,\frac{1-i}{2}\right)
+2Z\left(-i,1+i\right) + 2Z\left(i,1-i\right)\\
&\qquad +2\left[\frac{z-1}{z-i},z,1+i\right] -2\left[\frac{z-1}{z+i},z,1-i\right].
\intertext{As terms of the form
$[-,-,\zeta]$ for a primitive $n^{th}$ root of unity are
$n$-torsion, only the following terms survive multiplication by $2$, while the last two terms are equal to $2[\frac{z-1}{z-i},z,\frac{1+i}{1-i}]$:}
16C_{-1} -16C_1 &= 4Z\left(i,\frac{1+i}{2}\right)+4Z\left(-i,\frac{1-i}{2}\right)
+4Z\left(-i,1+i\right) + 4Z\left(i,1-i\right). \\
\intertext{Now we make use of \eqref{eq:zs} several times to
obtain:}
 &= 4Z(i,1) + 4Z(-i,1) = 0.
\end{align*}
So if we combine Levine's result with the relation $2C_1 = -4C_{-1}$ coming from the distribution relation, we immediately see that
$$24C_1 = 48C_{-1} = 0 \in CH^2(\mathbbm{Q}(i),3),$$ and because of the injectivity proved in Levine's
theorem, it is clear that $24C_1=0$ in $CH^2(\MQ,3)$ as well, hence the order of $C_1$ divides $24$. On the contrary, to see that the order is divisible by $24$, apply $-\Phi_{2,3}$ to get $Li_2(1)=\frac{\pi^2}{6}$. Since we know that the order of this group is exactly $24$, this makes $C_1$ a generator of $CH^2(\mathbbm{Q},3)$. Thus we can prove:
\begin{prop}
 The group $CH^2(\mathbbm{Q},3)\cong\mathbbm{Z} / 24\mathbbm{Z}$ is generated by the cycle $C_1\in CH^2(\mathbbm{Q},3)$. 
\end{prop}

\subsection{$CH^2(\MQ(i),3)$:} From Weibel's article one obtains that $CH^2(\mathbbm{Q}(i),3)\cong\mathbbm{Z}\oplus\mathbbm{Z} / 24\mathbbm{Z}$. Invoking the fact that $C_1\in CH^2(\MQ(i),3)$ still has order 24 by computing the image under the Abel -- Jacobi map, we immediately see that $C_1$ also generates the torsion part of the higher Chow group of $\MQ(i)$:

\begin{prop}
 The torsion part of the group $CH^2(\mathbbm{Q}(i),3)\cong\mathbbm{Z}\oplus\mathbbm{Z} / 24\mathbbm{Z}$ is generated by the cycle $C_1\in CH^2(\mathbbm{Q}(i),3)$.

The free part is generated by $4C_i$ or equivalently $4C_{-i}$.
\end{prop}
\begin{proof}
 The first part is clear and the second assertion follows from the fact that $\partial(C_i) = (1-i,i)$ and that $4(i)\in \partial Z^1(\mathbbm{Q}(i),2)$ so that not $C_i$ but $4C_i\in CH^2(\mathbbm{Q},3)$ and further that the image of $4C_i$ under the Abel -- Jacobi map has a non-vanishing imaginary part indicating that $4C_i$ is nontorsion in the Chow group. The same reasoning applies to $C_{-i}$.
\end{proof}

\begin{rem}
 The group $CH^2(\mathbbm{Q}(\zeta_3),3)\cong\mathbbm{Z}^2\oplus \mathbbm{Z} / 24\mathbbm{Z}$ is treated analogously: Since an generator of the torsion part of the $K$--group needs to be exactly $24$-torsion, and since a similar calculation as above shows that the cycle $3(C_\zeta + C_{\bar\zeta})$ is $12$-torsion, $C_1$ can be the only generator of the torsion part of $CH^2(\mathbbm{Q}(\zeta_3),3)$.

The free part of the Chow group is generated by $3C_{\zeta_3}$ resp. $3C_{\zeta^2_3}$ as can be seen by computing the Abel -- Jacobi map again.
\end{rem}

\begin{rem}
 Consulting Weibel's article again, esp. propositions 2.7 and 2.8, one can compute the numbers $w_2(F)$ for number fields rather easily and one can also see that the number very often is equal to 24. The theorem of Levine \ref{prop:levine} combined with the Abel -- Jacobi map of Kerr, Lewis and M\"uller-Stach implies that whenever the torsion part of $CH^2(F,3)$ for a number field $F$ has order 24, it is generated by $C_1$.
\end{rem}

\subsection{$CH^2(\MQ(\zeta_5),3)$:} A number field with a more interesting Chow group is treated in this example: We are looking for a generator of the torsion part of the second higher Chow group with order equal to $120$. We use the distribution relation for the fifth roots of unity:
\begin{align*}
 5C_1 &= 25C_1 + 25C_{\zeta_5} + 25C_{\zeta^2_5} +25C_{\bar\zeta^2_5} + 25C_{\bar\zeta_5}\\
-20C_1 &= 25C_{\zeta_5} + 25C_{\zeta^2_5} +25C_{\bar\zeta^2_5} + 25C_{\bar\zeta_5}\\
-150(C_{\zeta_5} + C_{\bar\zeta_5}) &=150(C_{\zeta^2_5} + C_{\bar\zeta^2_5}).
\end{align*}
Now, we use the inversion relation from proposition \ref{prop:inv} with $a=\zeta^3_5, b=\zeta^2_5$: We can already simplify the relation a little: Terms with a fifth root of unity in the right coordinate are $5$-torsion, while terms of the form $\left[\frac{z-1}{z-a},z,1-a\right]$ are $2$-torsion. Note also that $(\frac{1-\zeta_5^3}{1-\zeta_5^2})^{10} = 1$. So we multiply the whole relation by $10$ and note that $\bar\zeta^2_5=\zeta^3_5$ to get
\begin{align*}
&20 C_{\zeta_5} + 20C_{\bar\zeta_5} - 16C_1 = 10\left(2Z(\zeta^3_5,1-\zeta^3_5)+2Z\left(\zeta^2_5,1-\zeta^2_5\right)\right. \\
 &\left.\quad -Z\left(\zeta^3_5,\frac{1}{1-\zeta^3_5}\right) - Z\left(\zeta^2_5,\frac{\zeta^3_5}{\zeta^3_5-1}\right) - Z\left(\zeta^2_5,\frac{1}{1-\zeta^2_5}\right) - Z\left(\zeta^3_5,\frac{\zeta^2_5}{\zeta^2_5-1}\right)\right).
\end{align*}
\begin{rem}
 Note that we have already settled the admissibility of $Z(a,b)$ for arbitrary $a, b\in\Fs$. So one does not need to care when specializing to some primitive roots of unity.
\end{rem}
Now one applies equation \eqref{eq:zs} several times and keeps in mind that terms with a $n^{th}$ root of unity in the rightmost coordinate are $n$-torsion. In particular, e.g. $5Z\left(\zeta^2_5,\frac{1}{1-\zeta^2_5}\right)=-5Z\left(\zeta^2_5,1-\zeta^2_5\right)$ :
\begin{align*}
&20 C_{\zeta_5} + 20C_{\bar\zeta_5} - 16C_1 = 10\left(3Z(\zeta^3_5,1-\zeta^3_5)+3Z\left(\zeta^2_5,1-\zeta^2_5\right)\right. \\
 &\left.\quad - Z\left(\zeta^2_5,\frac{\zeta^3_5}{\zeta^3_5-1}\right) -  Z\left(\zeta^3_5,\frac{\zeta^2_5}{\zeta^2_5-1}\right)\right).
\intertext{Now, again from \eqref{eq:zs} and from $\zeta^2_5=\frac{1}{\zeta^3_5}$ we see that $5Z\left(\zeta^2_5,\frac{\zeta^3_5}{\zeta^3_5-1}\right))=-5Z\left(\zeta^2_5,1-\zeta^2_5\right)):$}
 &20 C_{\zeta_5} + 20C_{\bar\zeta_5} - 16C_1 =40Z(\zeta^3_5,1-\zeta^3_5)+40Z\left(\zeta^2_5,1-\zeta^2_5\right).
\intertext{Now,}
&40Z(\zeta^2_5,1-\zeta^2_5)+40Z\left(\zeta^3_5,1-\zeta^3_5\right) = 40Z(\zeta^2_5,1-\zeta^2_5)+40Z\left(\zeta^3_5,\frac{\zeta_5^2-1}{\zeta^2_5}\right)\\
&\qquad = 40Z(\zeta^2_5,1-\zeta^2_5)+40Z\left(\zeta^3_5,1-\zeta_5^2\right)+40Z(\zeta^3_5,\zeta_{10})\\
&\qquad = 40Z(\zeta^3_5,\zeta_{10}).
\end{align*}
Now, we may assume both arguments of this term to be fifth roots of unity because of the even factor in front. So, finally we assure $25Z(\zeta'_5,\zeta_5)=0\in C^2(\mathbbm{Q}(\zeta_5),3)/\partial C^2(\mathbbm{Q}(\zeta_5),4)$ for two fifth roots of unity $\zeta_5,\zeta'_5$: From proposition \ref{prop:rules1}, equation (ii), we see that $$Z(\zeta'_5,\zeta_5^2)= 2Z(\zeta'_5,\zeta_5)+\left[\frac{z-\zeta_5^2}{z-\zeta_5},z,\zeta'_5\right].$$ Since the last term on the right hand side is $5$-torsion, we conclude by induction
$$0=Z(\zeta'_5,1)=5Z(\zeta'_5,\zeta_5)+ 5-\text{torsion} \Longleftrightarrow 0=25Z(\zeta'_5,\zeta_5).$$ As the least common multiple of $25$ and $40$ is $300$, we deduce from our computations above:
$$300 C_{\zeta_5} + 300C_{\bar\zeta_5} - 240C_1=0\in CH^2(\mathbbm{Q}(\zeta_5),3),$$
but by proposition \ref{prop:levine}, $24C_1=0\in CH^2(\mathbbm{Q}(\zeta_5),3)$. These two results together with the relation $150(C_{\zeta_5}+C_{\bar\zeta_5}) = -150(C_{\zeta^2_5}+C_{\bar\zeta^2_5}) \in C^2(\mathbbm{Q}(\zeta_5),3)/\partial C^2(\mathbbm{Q}(\zeta_5),4)$ from the distribution relation give us in the end $$-300(C_{\zeta_5}+C_{\bar\zeta_5}) = 300(C_{\zeta^2_5}+C_{\bar\zeta^2_5}) = 0.$$ Remembering that $C_{\zeta_5}\notin CH^2(\mathbbm{Q}(\zeta_5),3)$, but $5C_{\zeta_5}\in CH^2(\mathbbm{Q}(\zeta_5),3)$, we can deduce that the higher Chow cycle $5(C_{\zeta_5}+C_{\bar\zeta_5})$ or equivalently $5(C_{\zeta^2_5}+C_{\bar\zeta^2_5})$ is $60$-torsion.

Using the Abel-Jacobi map of section 3, we compute the image of this cycle in Deligne cohomology $H^1_{\mathcal{D}}(\mathsf{Spec}(\mathbbm{C})^{an},\mathbbm{Z}(2))\cong \mathbbm{C}/\mathbbm{Z}(2) \cong \mathbbm{C} / 4\pi^2\mathbbm{Z}$ to be $\pi^2/15$:
$$-\Phi_{2,3}(5(C_{\zeta_5}+C_{\overline{\zeta}_5}))=5(Li_2(\zeta_5)+Li_2(\frac{1}{\zeta_5}))=5(-\frac{\pi^2}{6}-\frac{1}{2}\log(-\zeta_5)^2)=\frac{\pi^2}{15},$$ so that the order of this cycle is exactly $60$. Remembering that $C_1\in CH^2(\mathbbm{Q}(\zeta_5),3)$ is of order $24$, we deduce analogously to the former examples:

\begin{prop}
 The torsion part of the group $CH^2(\mathbbm{Q}(\zeta_5),3)\cong\mathbbm{Z}^2\oplus \mathbbm{Z} / 120\mathbbm{Z}$ is generated by the cycle $C_1+5(C_{\zeta_5}+C_{\bar\zeta_5}) \in CH^2(\mathbbm{Q}(\zeta_5),3)$. 

\end{prop}

\begin{rem}
 By class field theory one knows, that $\mathbbm{Q}(\sqrt{5}) \hookrightarrow \mathbbm{Q}(\zeta_5)$. One may ask for a generator of the higher Chow group of the former field coming from a generator of the higher Chow group of the latter field: One knows by the theorem of Levine that taking Galois conjugates of the generators of $CH^2(\mathbbm{Q}(\zeta_5),3)$ and pushing them down to $CH^2(\mathbbm{Q}(\sqrt{5}),3)$ via the canonical inclusion $\mathbbm{Q}(\sqrt{5}) \hookrightarrow \mathbbm{Q}(\zeta_5)$ that there is a linear combination of cycles in $CH^2(\mathbbm{Q}(\sqrt{5}),3)$ equivalent to the ones ``coming down'' generating this group, but there is no obvious way of constructing them.

For example, one can make use of the Abel-Jacobi map once again to check that the image of the cycle $$(C_{\frac{1}{2}(\sqrt{5}-1)} - C_{\frac{1}{4}(\sqrt{5}-1)^2})\in CH^2(\mathbbm{Q}(\sqrt{5}),3)$$ in Deligne cohomology is equal to $\pi^2/30$ showing that the order of this cycle is at least $120$. Combined with the result above on the isomorphic cyclotomic field, its order must be equal to $120$ turning it into a generator of $CH^2(\mathbbm{Q}(\sqrt{5}),3)$. However without the aid of \cite{Lew:1982} and the regulator map of \cite{KLM04} there would be no way of detecting cycles of this kind.
\end{rem}

\subsection{$CH^2(\MQ(\zeta_8),3)$:}
As a last example we consider another cyclotomic field containing three quadratic subfields. In order to find a generator of its Chow group, we shall start with a distribution relation again. $$8C_1 = 64\left(C_{\zeta_8}+C_{i}+C_{\zeta^3_8}+C_{-1}+C_{\zeta^5_8}+C_{-i}+C_{\zeta^7_8}+C_{1}  \right).$$ Using the distribution relation for the fourth roots of unity, i.e. $$4C_1 = 16(C_i+C_{-1}+C_{-i}+C_1),$$ we deduce
\begin{align*}
-8C_1 &=  64\left(C_{\zeta_8}+C_{\zeta^3_8}+C_{\zeta^5_8}+C_{\zeta^7_8}  \right)\\
-192(C_{\zeta_8} +C_{\zeta^7_8}) &= 192(C_{\zeta^3_8} +C_{\zeta^5_8}).
\end{align*}
Let us now try to relate the terms in brackets using the inversion relation and some auxiliary relations between terms with a constant on the right. Note in particular that we already multiplied by $8$ in order to kill torsion terms:
\begin{align*}
16C_{\zeta_8} + & 16C_{\zeta^7_8} - 8 C_1 = 8Z(\zeta^n_8,1-\zeta^n_8)+8Z(\zeta^{n-1}_8,1-\zeta^{n-1}_8) + 8Z(\bar\zeta^n_8,1-\bar\zeta^n_8)\\ & +8Z(\bar\zeta^{n-1}_8,1-\bar\zeta^{n-1}_8)
-8Z\left(\bar\zeta^{n-1}_8,\frac{1}{1-\zeta^n_8}\right)-8Z\left(\zeta^{n-1}_8,\frac{1}{1-\bar\zeta^n_8}\right)\\ &-8Z\left(\bar\zeta^n_8,\frac{1}{1-\zeta^{n-1}_8}\right)-8Z\left(\zeta^n_8,\frac{1}{1-\bar\zeta^{n-1}_8}\right)
-8\left[\frac{z-1}{z-\zeta^n_8},z,1-\zeta^{n-1}_8\right]\\ &-8\left[\frac{z-1}{z-\zeta^{n-1}_8},z,1-\zeta^n_8\right].
\end{align*}
Now we simplify the terms above:
\begin{align*}
16C_{\zeta_8} +  & 16C_{\zeta^7_8} - 8 C_1 = -8Z\left(\zeta^n_8,(1-\zeta^n_8)(1-\bar\zeta^{n-1}_8)\right)-8Z(\bar\zeta^n_8,(1-\bar\zeta^n_8)(1-\zeta^{n-1}_8))\\
 &\qquad-8Z(\zeta^{n-1}_8,(1-\zeta^n_8)(1-\bar\zeta^{n-1}_8))-8Z(\bar\zeta^{n-1}_8,(1-\bar\zeta^n_8)(1-\zeta^{n-1}_8))\\
&\qquad-8\left[\frac{z-1}{z-\zeta^n_8},z,1-\zeta^{n-1}_8\right]-8\left[\frac{z-1}{z-\zeta^{n-1}_8},z,1-\zeta^n_8\right]\\
\intertext{and further}
&=-8Z(\zeta^n_8,(1-\zeta^n_8)(1-\bar\zeta^{n-1}_8))-8Z(\bar\zeta^n_8,(1-\zeta^n_8)(1-\bar\zeta^{n-1}_8))-8Z(\zeta_8^n,\bar\zeta_8)\\
 &\qquad-8Z(\zeta^{n-1}_8,(1-\zeta^n_8)(1-\bar\zeta^{n-1}_8))-8Z(\bar\zeta^{n-1}_8,(1-\zeta^n_8)(1-\bar\zeta^{n-1}_8))\\
&\qquad -8Z(\bar\zeta_8^{n-1},\bar\zeta_8)-8\left[\frac{z-1}{z-\zeta^n_8},z,1-\zeta^{n-1}_8\right]-8\left[\frac{z-1}{z-\zeta^{n-1}_8},z,1-\zeta^n_8\right]\\
&=-8Z(\zeta_8^n,\bar\zeta_8)-8Z(\bar\zeta_8^{n-1},\bar\zeta_8).
\end{align*}
The last step follows from the fact that
\begin{align*}
-8Z(\zeta^n_8,(1-\zeta^n_8)&(1-\bar\zeta^{n-1}_8))- 8Z(\bar\zeta^n_8,(1-\zeta^n_8)(1-\bar\zeta^{n-1}_8))  \\
&= 8\left[\frac{z-1}{z-\zeta^n_8},z,(1-\zeta^n_8)(1-\bar\zeta^{n-1}_8)\right] = 8\left[\frac{z-1}{z-\zeta^n_8},z,1-\bar\zeta^{n-1}_8\right]
\end{align*}
and analogously for the other term.
Finally it easy to see that $$0=8Z(\zeta_8^n,1) = 64Z(\zeta_8^n,\zeta_8^m)$$ since extra terms are $8$-torsion. So we conclude $128C_{\zeta_8} +  128C_{\bar\zeta_8} + 16 C_1 = 0$ or in other words $394(C_{\zeta_8}+C_{\bar\zeta_8}) = 0$. Again, not $C_{\zeta_8}$, but $8C_{\zeta_8}\in CH^2(\mathbbm{Q}(\zeta_8),3)$ so that the cycle $8(C_{\zeta_8}+C_{\bar\zeta_8})=8(C_{\zeta^3_8}+C_{\bar\zeta^3_8})$ is $48$-torsion.

Invoking a regulator argument as in the above case, we calculate the image of this cycle in $H^1_{\mathcal{D}}(\mathsf{Spec}(\mathbbm{Q})(\zeta_5))$, $\mathbbm{Z}(2))\cong \mathbbm{C} / 4\pi^2\mathbbm{Z}$ to be $\pi^2/12$, so that the order of this cycle is exactly $48$.

\begin{prop}
 The torsion part of the group $CH^2(\mathbbm{Q}(\zeta_8),3)\cong\mathbbm{Z}^2\oplus \mathbbm{Z} / 48\mathbbm{Z}$ is generated by the cycle $8(C_{\zeta_8}+C_{\bar\zeta_8})\in CH^2(\mathbbm{Q}(\zeta_5),3)$. 
\end{prop}

\begin{rem}
 One knows that $\mathbbm{Q}(\zeta_8)$ contains three quadratic subfields, namely $\mathbbm{Q}(\sqrt{2})$, $\mathbbm{Q}(\sqrt{-2})$, and $\mathbbm{Q}(i)$. It is easy to see that $4(C_{i}+C_{-i})\in CH^2(\mathbbm{Q}(\zeta_8),3)$ already lives in $CH^2(\mathbbm{Q}(i),3)$, and -- as we have seen -- also generates this group. Unfortunately, it is far more complicated to find generators of the other two quadratic subfields.

One needs to consider $Gal(\mathbbm{Q}(\zeta_8)\, |\, \mathbbm{Q}(\sqrt{\pm 2}))$ and take Galois conjugates of the generators of the higher Chow group of the cyclotomic field. A descent argument as above ensures that there is a cycle in the quadratic subfield generating its higher Chow group. Unfortunately this is not constructive. But with the help of \cite{Lew:1982} again, one finds that $$2C_{(\sqrt{2}-1)^4} - 12C_{(\sqrt{2}-1)^2} + 8C_{\sqrt{2}-1} = C_1.$$ Indeed, half of the left hand side is a cycle in $CH^2(\mathbbm{Q}(\sqrt{2}),3)$ having order $48$, i.e. is a generator of the Chow group.

A clever linear combination of terms evaluated at algebraic arguments certainly gives a generator of $CH^2(\mathbbm{Q}(\sqrt{-2}),3)$. But we did not try to determine this combination.
\end{rem}

\begin{rem}
The above results are based on the following reasoning: As one knows from \cite{Wei}, the group $K_3(F)$ for a number field $F$ with $r_1$ real and $r_2$ pairs of conjugate complex places is given by $K^{ind}_3(F)\cong \MZ^{r_2}\oplus\text{ torsion}$. It is also proved in this article that the torsion part of the $K$--group of a cyclotomic field $\MQ(\zeta_p)$, $p$ an odd prime, is the same as the one of the maximal real subfield $\MQ(\zeta_p+\overline{\zeta}_p)$. Thus combinations of Totaro cycles with cyclotomic arguments which already live in the maximal real subfield are most likely to generate the torsion part of $CH^2(F,3)$.
\end{rem}

\subsection{A remark on $\MQ_p$:}
Since our techniques apply to infinite fields, we can also consider $p$-adic fields: From \cite[Sect.5]{Wei}, one knows that if $E$ is a local field with residue field $\mathbbm{F}_q$ of characteristic $p$, then the Milnor $K$--groups are uncountable, uniquely divisible abelian groups for $n\geq 3$. Further, they are summands of the Quillen $K$--groups. These are known:

\begin{prop}{\cite[Prop. 5.3]{Wei}}
If $i > 0$ there is a summand of $K_{3}(E)$ isomorphic to $K_{3}(\mathbbm{F}_q)\cong\MZ / w_2(\mathbbm{F}_q)\mathbbm{Z}$, where $$w_2(\mathbbm{F}_q) = \begin{cases} 24, & q=2, q=3, \\ q^2-1, & else. \end{cases}$$ The complementary summand is uniquely $\ell$-divisible for every prime $\ell\neq p$, i.e., a $\MZ_{(p)}$-module.
\end{prop}

Since we are interested in the indecomposable part of $K_3(E)$, we have to determine the Milnor $K$--group $K_3^M(E)$. But as remarked in \cite[Sect. 5]{Wei}, this is only known to be an uncountable, uniquely divisible abelian group. Thus, we see that $$K_3^{ind}(E)\cong \MZ / w_2(E)\MZ\oplus M,$$ where $M$ denotes the quotient of a $\MZ_{(p)}$-module by the Milnor $K$--group. Further, $w_2(E)$ is divisible by $24$.

\begin{prop}
 Let $F=\MQ_p$ for $p=2,3,5$. Then there is a summand in the group $CH^2(F,3)$ generated by $C_1\in CH^2(F,3)$.
\end{prop}
\begin{proof}
 From the discussion above, the Chow group $CH^2(F,3)$ of the fields in question contains a finite direct summand of order $24$. We know that $C_1\in CH^2(\MQ,3)$ is of order $24$ and by Levine's theorem (proposition \ref{prop:levine}) above, its order cannot decrease in any extension.
\end{proof}


\end{document}